\documentclass{article}

\usepackage{hyperref}
\usepackage[bibtex-style]{amsrefs}
\usepackage{pstricks}
\usepackage{amsmath, amsthm, amssymb}
\usepackage{tikz}
\usepackage{wrapfig}
\usepackage{xspace}

\usetikzlibrary{patterns}
\usetikzlibrary{snakes}
\usetikzlibrary{shapes.geometric}

\newtheorem{theorem}{Theorem}
\newtheorem{corollary}[theorem]{Corollary}
\newtheorem{conjecture}[theorem]{Conjecture}
\newtheorem{proposition}[theorem]{Proposition}

\newtheorem{definition}[theorem]{Definition}

\newtheorem{observation}[theorem]{Observation}
\newtheorem{example}[theorem]{Example}

\newcommand{\bigclass}{{\mathcal M}}

\renewcommand{\S}{\mathbf{S}}
\newcommand{\R}{\mathbf{R}}
\newcommand{\A}{\mathbf{A}}
\newcommand{\B}{\mathbf{B}}
\newcommand{\Z}{\mathbf{Z}}
\renewcommand{\P}{P}

\newcommand{\avoids}{\mathop{\mathrm{Av}}}

\newcommand{\Av}{\avoids}

\newcommand{\post}{\mathop{\mathrm{\bf Post}}}
\newcommand{\inreading}{\mathop{\mathrm{\bf In}}}
\newcommand{\InTree}{\mathop{\mathrm{T_{in}}}}
\newcommand{\canT}{\mathop{\mathcal{T}}}

\newcommand{\Push}{\ensuremath{\mathbf{Push}}\xspace}
\newcommand{\Pop}{\ensuremath{\mathbf{Pop}}\xspace}

\DeclareMathOperator{\zeil}{\mathrm{zeil}}
\DeclareMathOperator{\Rzeil}{\mathrm{Rzeil}}

\newcounter{indice}

\newcommand{\permutation}[1]{
\setcounter{indice}{0};
\foreach \i in {#1}
\addtocounter{indice}{1};

\addtocounter{indice}{1}
\draw [help lines] (1,1) grid (\theindice,\theindice);

\setcounter{indice}{1};

\foreach \i in { #1 } {
\draw (\theindice+.5,\i+.5) [fill] circle (.2);
\addtocounter{indice}{1};
}
\addtocounter{indice}{-1};
}

\title{Operators of equivalent sorting power and related Wilf-equivalences}
\author{Michael Albert\thanks{Department of Computer Science, University of Otago, Dunedin, New Zealand, {\tt malbert@cs.otago.ac.nz}} \and Mathilde Bouvel\thanks{Institut f\"ur Mathematik, Universit\"at Z\"urich, Switzerland, {\tt mathilde.bouvel@math.uzh.ch}}}
\date{}

\begin{document}
\maketitle

\begin{abstract}
We study sorting operators $\mathbf{A}$ on permutations that are obtained composing Knuth's stack sorting operator $\mathbf{S}$ and the reversal operator $\mathbf{R}$, as many times as desired. For any such operator $\mathbf{A}$, we provide a size-preserving bijection between the set of permutations sorted by $\mathbf{S} \circ \mathbf{A}$ and the set of those sorted by $\mathbf{S} \circ \mathbf{R} \circ \mathbf{A}$, proving that these sets are enumerated by the same sequence, but also that many classical permutation statistics are equidistributed across these two sets. 
The description of this family of bijections is based on a bijection between the set of permutations avoiding the pattern $231$ and the set of those avoiding $132$ which preserves many permutation statistics. We also present other properties of this bijection, in particular for finding pairs of Wilf-equivalent permutation classes.
\end{abstract}

\section{Introduction}

Partial sorting algorithms were one of the early motivations for the study of permutation patterns. In the late 1960s, Knuth~\cite{Knuth:Art} considered the problem of sorting a permutation of $[n] = \{1,2,\dots,n\}$ using only a stack. 
This problem takes a permutation $\pi$ in one line notation as input, starts with an empty stack, and its goal is to sort $\pi$ using only the \Push and \Pop operations. 
Knuth showed that a permutation $\pi = \pi(1) \pi(2) \ldots \pi(n)$ may be sorted by a stack if and only if the following procedure sorts $\pi$: 

For $i$ from $1$ to $n$\\ 
\hspace*{0.7cm} While the stack is nonempty and $\pi(i)$ is larger than the top of the stack, \\
\hspace*{1cm} \Pop to the output\\ 
\hspace*{0.7cm} \Push $\pi(i)$ on the stack\\
While the stack is nonempty, \\
\hspace*{0.7cm} \Pop to the output

Not all permutations are stack sortable, and Figure~\ref{FIG:Stack} shows an example of a permutation that fails to be sorted by a stack. 

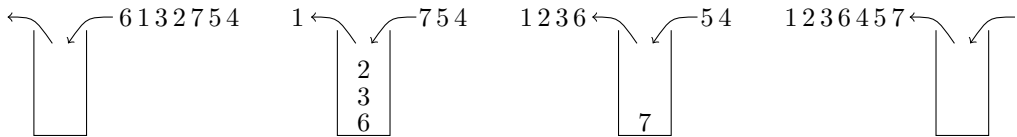
\begin{figure}[ht]
\begin{center}
\begin{tikzpicture}\begin{scope}[scale=0.35]
\draw (14,4) -- (14,0) -- (16,0)-- (16,4);
\draw[->] (17,4.5) .. controls (16,4.5) .. (15.3,3.5);
\draw[->] (14.7,3.5) .. controls (14,4.5) .. (13,4.5);
\draw (19.5,4.5) node {$6\, 1\, 3\, 2\, 7\, 5\, 4$};
\draw (15,-0.2) node {$~$};
\end{scope}
\end{tikzpicture} \quad \begin{tikzpicture}
\begin{scope}[scale=0.35]
\draw (14,4) -- (14,0) -- (16,0)-- (16,4);
\draw[->] (17,4.5) .. controls (16,4.5) .. (15.3,3.5);
\draw[->] (14.7,3.5) .. controls (14,4.5) .. (13,4.5);
\draw (15,0.5) node {$6$};
\draw (12.5,4.5) node {$1$};
\draw (15,1.5) node {$3$};
\draw (15,2.5) node {$2$};
\draw (18,4.5) node {$7\, 5\, 4$};
\draw (15,-0.2) node {$~$};
\end{scope}
\end{tikzpicture} \quad \begin{tikzpicture}
\begin{scope}[scale=0.35]
\draw (14,4) -- (14,0) -- (16,0)-- (16,4);
\draw[->] (17,4.5) .. controls (16,4.5) .. (15.3,3.5);
\draw[->] (14.7,3.5) .. controls (14,4.5) .. (13,4.5);
\draw (15,0.5) node {$7$};
\draw (11.5,4.5) node {$1\, 2\, 3\, 6$};
\draw (17.7,4.5) node {$5\, 4$};
\draw (15,-0.2) node {$~$};
\end{scope}
\end{tikzpicture} \quad \begin{tikzpicture}
\begin{scope}[scale=0.35]
\draw (14,4) -- (14,0) -- (16,0)-- (16,4);
\draw[->] (17,4.5) .. controls (16,4.5) .. (15.3,3.5);
\draw[->] (14.7,3.5) .. controls (14,4.5) .. (13,4.5);
\draw (10.5,4.5) node {$1\, 2\, 3\, 6\, 4\, 5 \, 7$};
\draw (15,-0.2) node {$~$};
\end{scope}
\end{tikzpicture}
\end{center}
\caption{Some steps of the stack sorting procedure applied to $\pi = 6\, 1\, 3\, 2\, 7\, 5\, 4$. \label{FIG:Stack}}
\end{figure}

Knuth~\cite{Knuth:Art} also characterized the permutations that may be sorted by a stack. A first way to present this characterization is as follows: if a permutation $\pi$ of $[n]$ is written in one line notation as $\alpha n \beta$, then $\pi$ is sortable if and only if: each of $\alpha$ and $\beta$ is sortable (thought of as permutations of the values they contain); and each value in $\alpha$ is less than any value in $\beta$ (or simply $\alpha < \beta$). The first condition is clearly necessary -- the second condition is also necessary as, when $n$ is the first element remaining to be added to the stack, the entire stack must be emptied to have any hope of success, otherwise $n$ will precede some other element in the output, and the output will not be sorted. That the conditions are sufficient is also clear -- the requisite operations are: sort and output $\alpha$; add $n$ to the stack; sort and output $\beta$; remove $n$ from the stack. 

Another classical characterization of stack sortable permutations is simply derived from the description above. Stack sortable permutations are those that may not contain subwords (not necessarily consecutive) of the form $bca$ where $a < b < c$. Such permutations are said to \emph{avoid the pattern} $231$, and the collection of all such is denoted $\Av(231)$. This result opened the way to the study of pattern avoidance in permutations. A permutation $\pi = \pi(1) \pi(2) \ldots \pi(k)$ is a \emph{pattern} of a permutation $\sigma = \sigma(1) \sigma(2) \ldots \sigma(n)$ when there exist $1 \leq i_1 < i_2 < \ldots < i_k \leq n$ such that $\pi$ is order isomorphic to $\sigma(i_1)\ldots \sigma(i_k)$. If $\pi$ is not a pattern of $\sigma$ then we say that $\sigma$ \emph{avoids} $\pi$. We denote by $\Av(B)$ the set of all permutations that avoid simultaneously all the patterns $\pi \in B$. 

The simple behavior explained by Knuth prompted many other investigations of stack sorting and its variations beginning with works by Pratt and Tarjan \cite{Pratt:Computing, Tarjan:Sorting}. 
In the 1990s, West \cite{West:Twice} described by the avoidance of generalized patterns the permutations that can be sorted using $\S \circ \S$, and Zeilberger \cite{Zeilberger:West} subsequently confirmed a conjecture of West's on their enumeration.  
A characterization of permutation sorted by $\S \circ \S \circ \S$ has recently been given by Claesson and \'Ulfarsson~\cite{Ulf12,ClUl12}. It involves even more general patterns, but does not allow the enumeration of permutations sorted by $\S^3$. Going further, the characterization and enumeration of permutations sorted by $\S^k$ for $k\geq 4$ are open questions. 

\bigskip

Instead of a procedure, stack sorting can equivalently be considered as an operator, $\S$, applied to permutations and defined recursively as: $\S(\alpha n \beta) = \S(\alpha) \S(\beta) n$. In this work, we shall take this point of view. 
We also adopt the viewpoint throughout that any sequence of distinct values can be interpreted as a permutation and ``$n$'' always denotes the maximum element of such a sequence. 

Bousquet-M\'{e}lou \cite{Bousquet:Sorted} considered the operator $\S$ and characterized, given $\pi$, the set $\S^{-1}(\pi)$. We shall be extending her results, and will discuss them in more detail later. As explained in Section~\ref{SEC:Mireille}, central to her analysis is the observation that the operator $\S$ can be described in the following terms: given a permutation $\pi$ form the unique decreasing binary tree $\InTree(\pi)$ whose in-order reading is $\pi$, then $\S(\pi)$ is the post-order reading of this tree.

A second operator on permutations is the reversal operator, that reads permutations from right to left -- it can also be modeled by using a stack where we are obliged to input the entire permutation to the stack before performing any output. The reversal operator, $\R$ is one of eight natural symmetries on the collection of permutations. Bouvel and Guibert \cite{Bouvel:Enumeration} considered the enumeration of permutations sorted by $\S \circ \R \circ \S$ as well as the sets defined similarly with other symmetries in place of $\R$. In experimental investigations aimed at providing extensions to their results they noticed an interesting phenomenon that can be expressed as: 

\begin{conjecture}
Take $\A$ to be any composition of the operators $\S$ and $\R$; then the number of permutations sorted by $\S \circ \A$ and by $\S \circ \R \circ \A$ is the same. Moreover, many permutation statistics are equidistributed across these two sets.
\label{CONJ:BouvelGuibert}
\end{conjecture}

It is the primary purpose of this article to prove that this is indeed the case, and the proof of Conjecture~\ref{CONJ:BouvelGuibert} will be given in Section~\ref{SEC:main}.

With the characterization of stack sortable permutations as $\Av(231)$, proving Conjecture~\ref{CONJ:BouvelGuibert} is equivalent to showing that there is a size-preserving bijection between the elements of $\Av(231)$ belonging to the image of $\A$, and the elements of $\Av(231)$ belonging to the image of $\R \circ \A$, with the additional condition that the bijection preserves the number of preimages under $\A$ (resp. $\R \circ \A$). Equivalently, we can replace this latter set by the elements of $\Av(132)$ belonging to the image of $\A$, since the self-inverse operator $\R$ immediately provides a bijection between $\Av(231)$ and $\Av(132)$. 

In establishing this result we make use of a very natural bijection -- denoted $\P$ -- between $\Av(231)$ and $\Av(132)$, which however rarely appears in the literature. As noticed in~\cite{DDJSS12}, this bijection preserves many permutation statistics, and we add more statistics to the list in Section~\ref{SEC:P}. 

Finally, in Section~\ref{SEC:Wilf}, we show how this bijection $\P$ can be used to derive Wilf-equivalences between some pairs of permutation classes of the form $\Av(231,\tau)$, for $\tau$ avoiding $231$. Namely, for every $n$, we describe $\lceil n/2 \rceil$ pairs of patterns $\tau$ and $\tau'$ such that $\R \circ \P$ is a bijection between $\Av(231,\tau)$ and $\Av(231,\tau')$. 

\section{Preimages of permutations in the image of $\S$}\label{SEC:Mireille}

As noted earlier, the description of the elements of $\S^{-1}(\pi)$ for $\pi$ in the image of $\S$ was carried out in \cite{Bousquet:Sorted}. This description is central to our work, so we review it here. 

\subsection{Some basics about binary trees and permutations}

A binary tree (whose internal vertices are labeled by integers) is \emph{decreasing} when $a < b$ for any child labeled by $a$ of a vertex labeled by $b$. 

The \emph{post-order reading} $\post$ is recursively defined by associating the empty word $\varepsilon$ to the empty tree $T_{\varepsilon}$, and the word $\post(T_\ell) \cdot \post(T_r) \cdot n$ to any non empty binary tree 
\raisebox{-1.5ex}{\begin{tikzpicture}
\begin{scope}
\tikzstyle{every child}=[level distance=3mm]
\tikzstyle{every node}=[inner sep=0,minimum size=0mm] 
\tikzstyle{level 1}=[sibling distance=10mm]
\node{\footnotesize $n$} child {node {$T_\ell$}} child { node {$T_r$}};
\end{scope}
\end{tikzpicture}
}. 

Similarly, the \emph{in-order reading} $\inreading$ is recursively defined by associating the empty word $\varepsilon$ to the empty tree $T_{\varepsilon}$, and the word $\inreading(T_\ell) \cdot n \cdot \inreading(T_r)$ to any non empty binary tree 
\raisebox{-1.5ex}{\begin{tikzpicture}
\begin{scope}
\tikzstyle{every child}=[level distance=3mm]
\tikzstyle{every node}=[inner sep=0,minimum size=0mm] 
\tikzstyle{level 1}=[sibling distance=10mm]
\node{\footnotesize $n$} child {node {$T_\ell$}} child { node {$T_r$}};
\end{scope}
\end{tikzpicture}
}.

\begin{observation} \label{OBS:Tin}
For any permutation $\pi$, there is a unique decreasing binary tree whose in-order reading is $\pi$. We denote it $\InTree(\pi)$. 
\end{observation}

Namely, $\InTree(\pi)$ is recursively described by $\InTree(\varepsilon) = T_{\varepsilon}$ and 
\[
\InTree(\alpha n \beta) = \raisebox{-1.5ex}{\begin{tikzpicture}
\begin{scope}
\tikzstyle{every child}=[level distance=1.5mm]
\tikzstyle{every node}=[inner sep=0,minimum size=0mm] 
\tikzstyle{level 1}=[sibling distance=12mm]
\node {\footnotesize $n$}
  child[child anchor=north] {node[draw,shape=isosceles triangle, shape border rotate=90,anchor=north, inner sep=0.1,isosceles triangle apex angle=60] {\scriptsize $\InTree(\alpha)$}}
  child[child anchor=north] {node[draw,shape=isosceles triangle, shape border rotate=90,anchor=north, inner sep=0.1,isosceles triangle apex angle=60] {\scriptsize $\InTree(\beta)$}};
\end{scope}
\end{tikzpicture}
} \text{ where } n = \max (\alpha n \beta)\text{.}
\]

\subsection{Trees and preimages of permutations in the image of $\S$}

As observed in~\cite[Proposition 2.1]{Bousquet:Sorted}, it may be deduced from the recursive definitions of $\InTree$, $\post$ and $\S$ given above (recall that $\S(\alpha n \beta) = \S(\alpha) \S(\beta) n$) that $\S$ converts in-order reading of decreasing binary trees to post-order reading: 

\begin{observation}
For any permutation $\pi$, the post-order reading of the in-order tree of $\pi$ is the result of applying the stack sorting operator to $\pi$, \emph{i.e.} $\post(\InTree(\pi))=\S(\pi)$.
\label{OBS:stack_sorting_on_trees}
\end{observation}

Notice furthermore for future use that, because every decreasing binary tree is the in-order tree of some permutation, Observation~\ref{OBS:stack_sorting_on_trees} implies that:

\begin{corollary}
The post-order reading of any decreasing binary tree is in the image of $\S$.
\label{COR:post_in_S}
\end{corollary}

From Observation~\ref{OBS:stack_sorting_on_trees}, for any permutation $\tau$ in the image of $\S$, describing $\S^{-1}(\tau)$ is equivalent to describing the decreasing binary trees, $T$, with post-order reading $\tau$. As~\cite{Bousquet:Sorted} shows, this set of trees may be characterized by a single tree associated with $\tau$. 

\begin{definition}
A decreasing binary tree is \emph{canonical} if it has the following property: any vertex, $z$, that has a left child, $x$, also has a right child, and the leftmost value $y$ in the subtree of the right child of $z$ is less than $x$. 
\end{definition}

\begin{proposition}[\cite{Bousquet:Sorted}, Proposition 2.6]
For any $\tau$ in the image of $\S$, there is a unique canonical tree, denoted $\canT_{\tau}$, with $\post(\canT_\tau) = \tau$. 
\label{PROP:unique_canonical_tree}
\end{proposition}

In fact, \cite{Bousquet:Sorted} also shows that the permutation $\pi$ obtained from the in-order reading of $\canT_{\tau}$ is the element of $\S^{-1}(\tau)$ having the greatest number of inversions. Moreover, \cite{Bousquet:Sorted} shows that all permutations in $\S^{-1}(\tau)$ (or equivalently, their in-order trees) may be recovered from $\canT_{\tau}$:

\begin{proposition}
Any decreasing binary tree whose post-order reading is $\tau$ (and only such trees) can be obtained from $\canT_\tau$ by a sequence of operations of the following type:  take a vertex $z$ with no left child, and one of its descendants $y$ on the leftmost branch of its right subtree; remove the subtree rooted at $y$ and make it the left subtree of $z$. 
\label{PROP:all_preimages_from_canonical_tree}
\end{proposition}

As shown in \cite[Proposition 2.7]{Bousquet:Sorted}, it follows in particular that $|\S^{-1}(\tau)|$ depends only on the structure of the tree $\canT_{\tau}$ and not on its labeling.

\begin{example}
The canonical tree associated with $\tau = 5 \, 1 \, 8 \, 2 \, 3 \, 6 \, 4 \, 7 \, 9$ is $T_{\tau} =$ \raisebox{-3.1ex}{\begin{tikzpicture}
\begin{scope}
\tikzstyle{every child}=[level distance=2.5mm]
\tikzstyle{every node}=[inner sep=0,minimum size=0mm] 
\tikzstyle{level 1}=[sibling distance=15mm]
\tikzstyle{level 2}=[sibling distance=8mm]
\tikzstyle{level 3}=[sibling distance=8mm]
\tikzstyle{level 4}=[sibling distance=8mm]
\node {\footnotesize $9$}
  child {node {\footnotesize $8$}
    child {node {\footnotesize $5$}
    }
    child {node {\footnotesize $1$}
    }
  }
  child {
    node {\footnotesize $7$}
      child {
        node {\footnotesize $6$}
          child [missing]
          child {
            node {\footnotesize $3$}
            child [missing]
            child {
              node {\footnotesize $2$}
            }
          }
      }
      child {
          node {\footnotesize $4$}
      }
   };
\end{scope}
\end{tikzpicture}}. Its in-order reading, $\pi = 5 \, 8 \, 1 \, 9 \, 6 \, 3 \, 2 \, 7 \, 4$ gives the permutation with the largest number of inversions subject to $\S(\pi) = \tau$. The four other decreasing binary trees with the same post-order reading are shown in Figure~\ref{FIG:example_canonical_tree}. Thus $|\S^{-1}(\tau)| = 5$. If the labels $8$ and $7$, and $5$ and $4$, were exchanged in the original tree, corresponding to $\tau' = 4 \, 1 \, 7 \, 2 \, 3 \, 6 \, 5 \, 8 \, 9$ then, because the tree is still canonical, the method for constructing permutations in $\S^{-1}(\tau')$ is still the same, and in particular $|\S^{-1}(\tau')| = |\S^{-1}(\tau)|$.
\end{example}

\begin{figure}[ht]
\begin{center}
\raisebox{1.8ex}{\begin{tikzpicture}
\begin{scope}
\tikzstyle{every child}=[level distance=2.5mm]
\tikzstyle{every node}=[inner sep=0,minimum size=0mm] 
\tikzstyle{level 1}=[sibling distance=15mm]
\tikzstyle{level 2}=[sibling distance=8mm]
\tikzstyle{level 3}=[sibling distance=8mm]
\tikzstyle{level 4}=[sibling distance=8mm]
\node {\footnotesize $9$}
  child {node {\footnotesize $8$}
    child {node {\footnotesize $5$}
    }
    child {node {\footnotesize $1$}
    }
  }
  child {
    node {\footnotesize $7$}
      child {
        node {\footnotesize $6$}
          child  {
              node {\footnotesize $2$}
            }
          child {
            node {\footnotesize $3$} 
          }
      }
      child {
          node {\footnotesize $4$}
      }
   };
\end{scope}
\end{tikzpicture}
}
\hspace{20pt}
\raisebox{0ex}{\begin{tikzpicture}
\begin{scope}
\tikzstyle{every child}=[level distance=2.5mm]
\tikzstyle{every node}=[inner sep=0,minimum size=0mm] 
\tikzstyle{level 1}=[sibling distance=15mm]
\tikzstyle{level 2}=[sibling distance=8mm]
\tikzstyle{level 3}=[sibling distance=8mm]
\tikzstyle{level 4}=[sibling distance=8mm]
\node {\footnotesize $9$}
  child {node {\footnotesize $8$}
    child {node {\footnotesize $5$}
    }
    child {node {\footnotesize $1$}
    }
  }
  child {
    node {\footnotesize $7$}
      child {
        node {\footnotesize $6$}
          child [missing]
          child {
            node {\footnotesize $3$}
            child {
              node {\footnotesize $2$}
            }
	    child [missing]
          }
      }
      child {
          node {\footnotesize $4$}
      }
   };
\end{scope}
\end{tikzpicture}
}
\hspace{20pt}
\raisebox{0ex}{\begin{tikzpicture}
\begin{scope}
\tikzstyle{every child}=[level distance=2.5mm]
\tikzstyle{every node}=[inner sep=0,minimum size=0mm] 
\tikzstyle{level 1}=[sibling distance=15mm]
\tikzstyle{level 2}=[sibling distance=8mm]
\tikzstyle{level 3}=[sibling distance=8mm]
\tikzstyle{level 4}=[sibling distance=8mm]
\node {\footnotesize $9$}
  child {node {\footnotesize $8$}
    child {node {\footnotesize $5$}
    }
    child {node {\footnotesize $1$}
    }
  }
  child {
    node {\footnotesize $7$}
      child {
        node {\footnotesize $6$}
          child {
            node {\footnotesize $3$}
            child [missing]
            child {
              node {\footnotesize $2$}
            }
          }
          child [missing]
      }
      child {
          node {\footnotesize $4$}
      }
   };
\end{scope}
\end{tikzpicture}
}
\hspace{20pt}
\raisebox{0ex}{\begin{tikzpicture}
\begin{scope}
\tikzstyle{every child}=[level distance=2.5mm]
\tikzstyle{every node}=[inner sep=0,minimum size=0mm] 
\tikzstyle{level 1}=[sibling distance=15mm]
\tikzstyle{level 2}=[sibling distance=8mm]
\tikzstyle{level 3}=[sibling distance=8mm]
\tikzstyle{level 4}=[sibling distance=8mm]
\node {\footnotesize $9$}
  child {node {\footnotesize $8$}
    child {node {\footnotesize $5$}
    }
    child {node {\footnotesize $1$}
    }
  }
  child {
    node {\footnotesize $7$}
      child {
        node {\footnotesize $6$}
          child {
            node {\footnotesize $3$}
            child {
              node {\footnotesize $2$}
            }
            child [missing]
          }
          child [missing]
      }
      child {
          node {\footnotesize $4$}
      }
   };
\end{scope}
\end{tikzpicture}
}
\end{center}
\caption{The four non canonical decreasing trees whose post-order reading is $\tau = 5 \, 1 \, 8 \, 2 \, 3 \, 6 \, 4 \, 7 \, 9$.} \label{FIG:example_canonical_tree}
\end{figure}
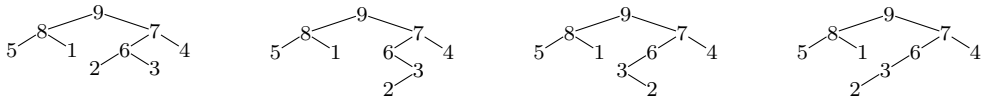

\section{A recursive bijection between $\Av(231)$ and $\Av(132)$}\label{SEC:P}

In this section we introduce a bijection, that we denote $\P$, between permutations in $\Av(231)$ and those in $\Av(132)$. 
Even though it is very naturally defined, this bijection seems to appear rather rarely in the literature -- only in~\cite{DDJSS12} to our knowledge. 

It is very easy to describe $\P$ recursively using the \emph{sum}, $\oplus$,  and \emph{skew sum}, $\ominus$, operations on permutations. These operations are easily understood on the \emph{diagrams} corresponding to permutations. The diagram of any permutation $\sigma$ of $[n]$ is the set of $n$ points in the plane at coordinates $(i, \sigma(i))$ -- see Figure~\ref{FIG:diagram} for some examples.
If $\alpha$ is a permutation of $[a]$ and $\beta$ of $[b]$ we define:
\vspace*{-0.3cm}
\begin{align*}
\alpha \oplus \beta &= \alpha \, (\beta + a) \text{ whose diagram is } \begin{tikzpicture}
\begin{scope}[scale=.2]
\useasboundingbox (0,1) rectangle (4,5);
\draw (0,0) rectangle (2,2);
\draw (1,1) node {\small $\alpha$};
\draw (2,2) rectangle (4,4);
\draw (3,3) node {\small $\beta$};
\end{scope}
\end{tikzpicture} \\
\alpha \ominus \beta &= (\alpha + b) \, \beta \text{ whose diagram is } \begin{tikzpicture}
\begin{scope}[scale=.2]
\useasboundingbox (0,1) rectangle (4,5);
\draw (0,2) rectangle (2,4);
\draw (1,3) node {\small $\alpha$};
\draw (2,0) rectangle (4,2);
\draw (3,1) node {\small $\beta$};
\end{scope}
\end{tikzpicture}\ .
\end{align*}
Here for example $\beta + a$ is just that sequence obtained by adding $a$ to every element of the sequence $\beta$ and \begin{tikzpicture}
\begin{scope}[scale=.2]
\useasboundingbox (0,0.5) rectangle (2,2.5);
\draw (0,0) rectangle (2,2);
\draw (1,1) node {\small $\alpha$};
\end{scope}
\end{tikzpicture} represents the diagram of permutation $\alpha$.

\begin{example}
Let $\alpha = 2\,3\,1$ and $\beta = 3\,1\,4\,2$. Then $\alpha \oplus \beta = 2\,3\,1\,6\,4\,7\,5$, while $\alpha\ominus \beta = 6\,7\,5\,3\,1\,4\,2$, as shown in Figure~\ref{FIG:diagram}.
\end{example}

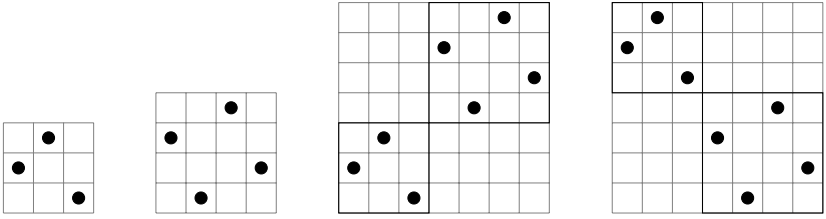
\begin{figure}[ht]
\begin{center}
\begin{tikzpicture}
\begin{scope}[scale=0.4]
\permutation{2,3,1}
\end{scope}
\end{tikzpicture} \qquad 
\begin{tikzpicture}
\begin{scope}[scale=0.4]
\permutation{3,1,4,2}
\end{scope}
\end{tikzpicture} \qquad 
\begin{tikzpicture}
\begin{scope}[scale=0.4]
\permutation{2,3,1,6,4,7,5}
\draw (1,1) rectangle (4,4);
\draw (4,4) rectangle (8,8);
\end{scope}
\end{tikzpicture} \qquad 
\begin{tikzpicture}
\begin{scope}[scale=0.4]
\permutation{6,7,5,3,1,4,2}
\draw (1,5) rectangle (4,8);
\draw (4,1) rectangle (8,5);
\end{scope}
\end{tikzpicture} \qquad 
\end{center}
\caption{From left to right, the diagrams of the permutations $\alpha = 2\,3\,1$, $\beta = 3\,1\,4\,2$, $\alpha \oplus \beta$ and $\alpha\ominus \beta$. \label{FIG:diagram}}
\end{figure}

Any $\pi \in \Av(231)$ is either the empty permutation $\varepsilon$ or has a unique decomposition in the form $\alpha \oplus (1 \ominus \beta)$ where $\alpha, \beta \in \Av(231)$ (and are possibly empty), and conversely any permutation of this latter form lies in $\Av(231)$. This is simply because the elements preceding the maximum in a $231$-avoiding permutation must all be less than those following the maximum, and the prefix before and suffix after the maximum must also avoid 231. Conversely, if a permutation has this structure it cannot involve 231. This decomposition makes it easy to define the bijection $\P$ recursively: $\P(\varepsilon)=\varepsilon$ and 
\[
\text{If } \pi = \alpha \oplus (1 \ominus \beta) \text{ then } \P(\pi) = (\P(\alpha) \oplus 1) \ominus \P(\beta).
\]
Alternatively, with diagrams:
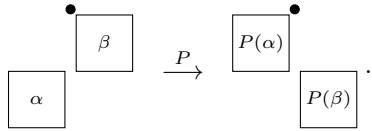
\[
\begin{tikzpicture}
\begin{scope}[scale=.3]
\useasboundingbox (-0.1,1) rectangle (5.6,5);
\draw (0,0) rectangle (2.5,2.5);
\draw (1.25,1.25) node {\scriptsize $\alpha$};
\draw (3,2.5) rectangle (5.5,5);
\draw (4.25,3.75) node {\scriptsize $\beta$};
\draw (2.75,5.25) [fill] circle (0.2);
\end{scope}
\end{tikzpicture}
\quad
\mbox{\raisebox{10pt}{$\stackrel{P}{\longrightarrow}$}}
\quad
\begin{tikzpicture}
\begin{scope}[scale=.3]
\useasboundingbox (-0.1,1) rectangle (5.6,5);
\draw (0,2.5) rectangle (2.5,5);
\draw (1.25,3.75) node {\scriptsize $\P(\alpha)$};
\draw (3,0) rectangle (5.5,2.5);
\draw (4.25,1.25) node {\scriptsize $\P(\beta)$};
\draw (2.75,5.25) [fill] circle (0.2);
\end{scope}
\end{tikzpicture}
\:
\text{\raisebox{12pt}{.}}
\]
As the $132$-avoiding permutations have a generic decomposition of the form shown on the right above, 
and since $\P(1) = 1$ maps the unique $231$-avoiding permutation of size $1$ to the unique $132$-avoiding permutation of size $1$, 
induction immediately implies that $\P: \Av(231) \to \Av(132)$ is a bijection. 

Let us introduce a notational convention that we shall use throughout. For any $\pi \in \Av(231)$, we can think of the sequence $P(\pi)$ as describing a relabeling of the values that occur in $\pi$ according to a certain permutation, denoted $\lambda_\pi$. Specifically, this means that $\lambda_\pi$ is defined by $P(\pi) = \lambda_\pi \circ \pi$. 

\begin{example}\label{EX:bijection_P}
For $\pi = 1 \, 5\, 3\, 2\, 4\, 9\, 8\, 6\, 7 \in \Av(231)$, we have $\P(\pi) = 7\, 8\, 5\, 4\, 6\, 9\, 3\, 1\, 2$. The corresponding relabeling $\lambda_\pi$ is $\lambda_\pi = 7\, 4\, 5\, 6\, 8\, 1\, 2\, 3\, 9$.
\end{example}

Recall from Section~\ref{SEC:Mireille} that $\InTree(\pi)$ is the decreasing binary tree whose in-order reading is $\pi$. It follows immediately by induction from the recursive description of $\P$ that:

\begin{observation}\label{OBS:P_preserves_Tin}
For any $\pi \in \Av(231)$, both $\InTree(\pi)$ and $\InTree(\P(\pi))$ have the same underlying unlabeled tree, or briefly ``$\P$ preserves the shape of in-order trees''. 
\end{observation}

\begin{figure}[ht]
\begin{center}
$\InTree(\pi)=$\begin{tikzpicture}
\begin{scope}
\tikzstyle{every child}=[level distance=2.5mm]
\tikzstyle{every node}=[inner sep=0,minimum size=0mm] 
\tikzstyle{level 1}=[sibling distance=15mm]
\tikzstyle{level 2}=[sibling distance=11mm]
\tikzstyle{level 3}=[sibling distance=8mm]
\tikzstyle{level 4}=[sibling distance=6mm]
\node {\footnotesize $9$}
  child {node {\footnotesize $5$}
    child {node {\footnotesize $1$}
    }
    child {node {\footnotesize $4$}
      child {node {\footnotesize $3$}
        child [missing]
        child {node {\footnotesize $2$}}
      }
      child [missing] 
    }
  }
  child {node {\footnotesize $8$}
    child [missing] 
    child {node {\footnotesize $7$}
      child {node {\footnotesize $6$}}
      child [missing] 
    }
  };
\end{scope}
\end{tikzpicture} 
\qquad 
$\InTree(\P(\pi))=$\begin{tikzpicture}
\begin{scope}
\tikzstyle{every child}=[level distance=2.5mm]
\tikzstyle{every node}=[inner sep=0,minimum size=0mm] 
\tikzstyle{level 1}=[sibling distance=15mm]
\tikzstyle{level 2}=[sibling distance=11mm]
\tikzstyle{level 3}=[sibling distance=8mm]
\tikzstyle{level 4}=[sibling distance=6mm]
\node {\footnotesize $9$}
  child {node {\footnotesize $8$}
    child {node {\footnotesize $7$}
    }
    child {node {\footnotesize $6$}
      child {node {\footnotesize $5$}
        child [missing]
        child {node {\footnotesize $4$}}
      }
      child [missing] 
    }
  }
  child {node {\footnotesize $3$}
    child [missing] 
    child {node {\footnotesize $2$}
      child {node {\footnotesize $1$}}
      child [missing] 
    }
  };
\end{scope}
\end{tikzpicture}
\end{center}
\caption{$\InTree(\pi)$ and $\InTree(\P(\pi))$ for the permutation $\pi = 1 \, 5\, 3\, 2\, 4\, 9\, 8\, 6\, 7$ of Example~\ref{EX:bijection_P}. \label{FIG:T_in}}
\end{figure}

Figure~\ref{FIG:T_in} shows an example. The acute reader will notice that in Figure~\ref{FIG:T_in}, not only $\InTree(\pi)$ and $\InTree(\P(\pi))$ have the same shape, but we also have $\lambda_\pi(\InTree(\pi)) = \InTree(\P(\pi))$. This is actually true in general, but the proof requires a bit more work (see the proof of Observation~\ref{OBS:identity} p.\pageref{OBS:identity}).

However, some nice properties of $\P$ in terms of permutation statistics follow from the simple fact that $\P$ preserves the shape of in-order trees. 

Recall that, for $\pi$ a permutation of $[n]$, a \emph{left-to-right} (resp.~\emph{right-to-left}) \emph{maximum} of $\pi$ is an element $\pi(i)$ such that for all $j<i$ (resp.~$j>i$), $\pi(j) < \pi(i)$, and that the \emph{up-down word} of $\pi$ is $w_{\pi} \in \{u,d\}^{n-1}$ with $w_{\pi}(i) = u$ (resp.~$d$) if $\pi(i) < \pi(i+1)$ (resp.~$\pi(i) > \pi(i+1)$).

\begin{observation}
For any permutation, the shape of its in-order tree determines the number and positions of its right-to-left maxima, the number and positions of its left-to-right maxima and its up-down word. 
\label{OBS:stat_on_InTree}
\end{observation}

\begin{proof}
Let $\pi$ by any permutation. 
That $w_{\pi}$ is determined by the shape of $\InTree(\pi)$ follows immediately by induction, from the recursive definition of the in-order reading $\inreading$. 
And the right-to-left (resp. left-to-right) maxima of $\pi$ correspond to the vertices lying on the right (resp. left) branch from the root of $\InTree(\pi)$, yielding the conclusion. 
\end{proof}

Observations~\ref{OBS:P_preserves_Tin} and~\ref{OBS:stat_on_InTree} then give that:

\begin{corollary}
$\P$ preserves the following statistics: the number and positions of the right-to-left maxima, the number and positions of the left-to-right maxima and the up-down word.
\end{corollary}

As noted earlier, $\P$ has already been used in the study of permutation statistics. Namely, \cite[proof of Theorem 2.6]{DDJSS12} shows that $\P$ preserves the descent set, and hence the major index. However, many other classical permutation statistics are also preserved by $\P$, namely all those that depend only on the up-down word, for instance the descent set, the major index, the number of peaks. Among all the statistics reported in~\cite[Section 2]{ClKi08}, the only ones that are preserved by $\P$ are the ones that depend only on the shape of in-order trees.

\section{Proof of Conjecture~\ref{CONJ:BouvelGuibert}}
\label{SEC:main}

\subsection{Preparation}\label{SEC:prep}

In addition to the results of Section~\ref{SEC:Mireille}, the principal ingredients in the proof to follow are a pair of observations concerning $\P$ and operators $\A$ which are compositions of $\S$ and $\R$.

\begin{observation}
\label{OBS:SR}
Let $\tau$ be any permutation, and $\A$ be any composition of the operators $\S$ and $\R$. Suppose that $x, y \in [n]$ and that in $\tau$ there are no values larger than $\max(x,y)$ occurring between $x$ and $y$. Then the same holds in $\A(\tau)$.
\end{observation}

\begin{proof}
It suffices to prove the result for $\S$ and $\R$ individually. For $\R$ it is trivial since the elements between $x$ and $y$ in $\tau$ and $\R(\tau)$ are the same. But for $\S$ it follows immediately from the recursive description: $\S(\tau) = \S(\alpha n \beta) = \S(\alpha) \S(\beta) n$. If one of $x$ or $y$ is $n$ then there is nothing to prove, while if not then they must both occur in $\alpha$ or in $\beta$ and the result follows by induction.
\end{proof}

For the second observation, recall that, for any $\pi \in \Av(231)$, we denote $\lambda_\pi$ the permutation such that $P(\pi) = \lambda_\pi \circ \pi$, and that we view it as a relabeling of the elements of $\pi$. 

\begin{observation}
\label{OBS:P}
Let $\pi \in \Av(231)$ be given and suppose that $x, y \in [n]$, $x < y$, and in $\pi$ there are no values larger than $y$ 
occurring between $x$ and $y$. Then $\lambda_{\pi}(x) < \lambda_{\pi}(y)$.
\end{observation}

In other words, Observation~\ref{OBS:P} simply says that $\lambda_\pi$ preserves the ordering among elements of $\pi$ which do not contain a larger element between them. 

\begin{proof}
The key argument is that, from the construction of $P$, the only way that one element can be moved above another one is to (at some point in the recursion) have a larger element in between. This can be expressed formally by induction, using the recursive definition of~$\P$. 

Let $\pi = \alpha n \beta = \alpha \oplus (1 \ominus \beta)$ and let $a=|\alpha|$ and $b = |\beta|$ (so $a + b = n-1$). If $y = n$ the result is trivial as $\lambda_{\pi}$ fixes $n$. Otherwise $x, y \in \alpha$ or $x, y \in \beta$, as by assumption they have no larger element between them in $\pi$. In the first case we have $x=\pi(i) = \alpha(i)$ and $y=\pi(j) = \alpha(j)$ for $i, j \in [a]$, and in the second case we have $x=\pi(i+a+1) = \beta(i) + a$ and $y=\pi(j+a+1) = \beta(j) + a$ for $i,  j \in [b]$.
It follows by induction that (depending on which case applies) $\lambda_{\alpha}(x) < \lambda_{\alpha}(y)$ or $\lambda_{\beta}(x-a) < \lambda_{\beta}(y-a)$. 

Note that $\lambda_{\pi}$ sends every $\ell \in [a]$ to $b + \lambda_{\alpha}(\ell)$ and every $a+\ell$ for $\ell \in [b]$ to $\lambda_{\beta}(\ell)$. Hence, when $x$ and $y$ are in $\alpha$ (resp.~$\beta$) we deduce from $\lambda_{\alpha}(x) < \lambda_{\alpha}(y)$ (resp.~$\lambda_{\beta}(x-a) < \lambda_{\beta}(y-a)$) that $\lambda_{\pi}(x) < \lambda_{\pi}(y)$, concluding the proof. 
\end{proof}

\subsection{The main argument}

In this section we prove the main result. Recall that $\A$ is an operator formed by some composition of $\S$ and $\R$. For any such operator, we shall write $\pi \in \A$ to denote that $\pi$ is in the image of $\A$.

As above, for any $\pi \in \Av(231)$, we consider $\lambda_{\pi}$ as a relabeling of the elements of $[n]$.
We extend its effect to permutations, trees etc.~that carry labels from $[n]$: applying $\lambda_{\pi}$ to such an object will simply mean to apply $\lambda_{\pi}$ to each of its labels.

\begin{definition} 
We define a function $\Phi_{\A}$ from the set of permutations sorted by $\S \circ \A$ to the set of all permutations as follows. For $\theta$ a permutation sorted by $\S \circ \A$, since $\A(\theta) \in \Av(231)$, we have $\lambda_{\A(\theta)}$ defined by $\P(\A(\theta)) = \lambda_{\A(\theta)} \circ \A(\theta)$ and we then set $\Phi_{\A}(\theta)= \lambda_{\A(\theta)} \circ \theta$.
\end{definition} 

In other words $\Phi_{\A}$ relabels a permutation $\theta$ sorted by $\S \circ \A$ in the same way that $\A(\theta)$ is relabeled to produce $P(\A(\theta))$. We will prove (see Corollary~\ref{COR:bijection}) that $\Phi_{\A}$ is a bijection from the set of permutations sorted by $\S \circ \A$ to the set of those sorted by $\S \circ \R \circ \A$. The key to this argument of course is to establish that $\A(\Phi_{\A}(\theta)) = P(\A(\theta))$.

\begin{definition}
An operator $\A$ which is a composition of $\S$ and $\R$ \emph{respects $\P$} if it has the following property (illustrated in Figure~\ref{FIG:A_respects_P}):

For each $\pi \in \Av(231) \cap \A$, 
\begin{itemize}
\item
For each $\theta$ such that $\A(\theta) = \pi$, we have $ \A(\Phi_{\A}(\theta)) = \P(\pi) = \lambda_{\pi} \circ \pi$ and \\ $\InTree(\Phi_{\A}(\theta)) = \lambda_{\pi}(\InTree(\theta))$, and
\item
the correspondence $\Phi_{\A}: \theta \mapsto \Phi_{\A}(\theta)$ is a bijection between $\A^{-1}(\pi)$ and $\A^{-1}(\P(\pi))$.
\end{itemize}
\end{definition}

In the above, notice that because $\A(\theta) = \pi$ we actually have $\Phi_{\A}(\theta) = \lambda_{\pi} \circ \theta$.

\begin{figure}[ht]
\begin{center}
\begin{tikzpicture}
\begin{scope}[scale=1]
\node at (-0.3,1) {\small $\{12\dots n\}$};
\node at (1.1,1) {$\xleftarrow{\hspace*{1.2cm}}$};
\node at (1.2,1.3) {$\S$};
\filldraw[fill=gray!15!white, draw=black] (4,1) arc (0:360:1cm and 0.4cm);
\node at (3,1.7) {\small {$\Av(231) \cap \A$}};
\node[anchor=east] at (3.9,1) {{\small $\pi$}};
\filldraw[fill=gray!15!white, draw=black] (7.6,1) arc (0:360:1.3cm and 0.45cm);
\node at (4.5,1) {$\xleftarrow{\hspace*{1.2cm}}$};
\node at (4.4,1.3) {$\A$};
\node[anchor=west] at (5.1,1) {\small {$\theta$}};
\node at (6.3,1.7) {\small {$\A^{-1}(\pi)$}};

\node at (-0.3,0) {\small $\{12\dots n\}$};
\node at (0.75,0) {$\xleftarrow{\hspace*{0.5cm}}$};
\node at (0.75,0.3) {$\S$};
\node at (1.6,0) {$\xleftarrow{\hspace*{0.5cm}}$};
\node at (1.6,0.3) {$\R$};
\filldraw[fill=gray!15!white, draw=black] (4,0) arc (0:360:1cm and 0.4cm);
\node at (3,-0.7) {\small {$\Av(132) \cap \A$}};
\node[anchor=east] at (3.9,0) {{\small $\P(\pi)$ \hspace*{-0.7em} $=$ \hspace*{-0.7em} $\lambda_{\pi}$ \hspace*{-0.7em} $\circ$ \hspace*{-0.7em} $\pi$}};
\filldraw[fill=gray!15!white, draw=black] (7.6,0) arc (0:360:1.3cm and 0.45cm);
\node[anchor=west] at (5.1,0) {{\small $\Phi_{\A}(\theta) = \lambda_{\pi} \circ \theta$}};
\node at (4.5,0) {$\xleftarrow{\hspace*{1.2cm}}$};
\node at (4.4,0.3) {$\A$};
\node at (6.3,-0.7) {\small {$\A^{-1}(\P(\pi))$}};

\draw[<->] (7.7,1) .. controls +(right:1cm) and +(right:1cm) .. (7.7,0);
\node at (9.1,0.7) {\small {$\Phi_{\A}$ is a }};
\node at (9.1,0.3) {\small {bijection}};
\end{scope}
\end{tikzpicture}
\end{center}
\caption{Overview of an operator $\A$ that respects $\P$. \label{FIG:A_respects_P}}
\end{figure}
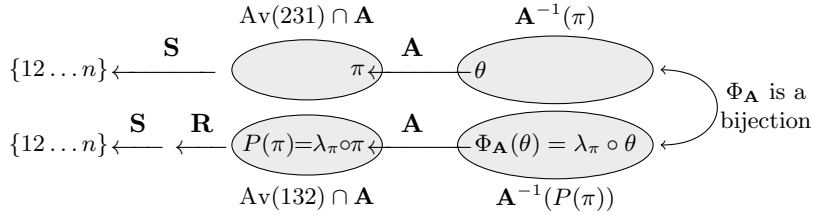

\begin{proposition}
If $\A$ respects $\P$ then so does $\A \circ \R$.
\label{PROP:AoR_respects_P}
\end{proposition}

\begin{proof}
Let $\pi \in \Av(231) \cap (\A \circ \R)$ and $\theta$ be such that $(\A \circ \R)(\theta) = \pi$. Let $\tau = \R(\theta)$. Then $\A(\tau) = \pi$ and since $\A$ respects $\P$, $\A(\Phi_{\A}(\tau)) = \P(\pi)$ and $\InTree(\Phi_{\A}(\tau)) = \lambda_{\pi}(\InTree(\tau))$. 

Because $\R$ is an involution on permutations that acts only on positions whereas $\lambda_{\pi}$ acts on values only, we prove that $\R\big(\Phi_{\A \circ \R}(\theta)\big) = \Phi_{\A}(\tau)$. Indeed, for any $i \in [n]$, we have:
\[
\R\big(\Phi_{\A \circ \R}(\theta)\big) (i) = \Phi_{\A \circ \R}(\theta) (n+1-i) = \lambda_{\A \circ \R(\theta)}\big(\theta(n+1-i)\big) = \lambda_{\A(\tau)}\big(\tau(i)\big) = \Phi_{\A}(\tau)(i).
\]
It follows that $(\A \circ \R)\left( \Phi_{\A \circ \R}(\theta) \right) = \A(\Phi_{\A}(\tau)) = \P(\pi)$.

Moreover, applying $\R$ to a permutation is equivalent to recursively exchanging left and right subtrees in its in-order tree. Because we have  $\R \left( \Phi_{\A}(\tau) \right) = \Phi_{\A \circ \R}(\theta)$ and $\R(\tau) = \theta$, we deduce from $\InTree(\Phi_{\A}(\tau)) = \lambda_{\pi}(\InTree(\tau))$ that $\InTree(\Phi_{\A \circ \R}(\theta)) = \lambda_{\pi}(\InTree(\theta))$. 

Finally, the correspondence $\Phi_{\A \circ \R}: \theta \mapsto \Phi_{\A \circ \R}(\theta)$ is a composition of three bijections: first $\R$ from $(\A \circ \R)^{-1}(\pi)$ to $\A^{-1}(\pi)$, then $\Phi_{\A}$ from $\A^{-1}(\pi)$ to $\A^{-1}(\P(\pi))$ since $\A$ respects $\P$, and last $\R^{-1}=\R$ again from $\A^{-1}(\P(\pi))$ to $( \A \circ \R)^{-1}(\P(\pi))$. This proves that $\Phi_{\A \circ \R}$ is a bijection between $(\A \circ \R)^{-1}(\pi)$ and $( \A \circ \R)^{-1}(\P(\pi))$ and concludes the proof that $\A \circ \R$ respects $\P$.
\end{proof}

\begin{proposition}
If $\A$ respects $\P$ then so does $\A \circ \S$.
\label{PROP:AoS_respects_P}
\end{proposition}

\begin{proof}
In this case, the argument is a little more involved. 

Let $\pi \in \Av(231) \cap (\A \circ \S)$. Throughout the proof, let us denote the relabeling accomplished by $\lambda_\pi$ by a primed symbol, i.e.~$w'$ represents the effect of relabeling $w$ by $\lambda_\pi$ for any entity $w$. For instance, $\P(\pi) = \lambda_\pi \circ \pi = \pi'$.

Consider $\theta \in (\A \circ \S)^{-1}(\pi)$, and define $\tau =\S(\theta)$. Then, we have $\tau \in \S$ and $\A(\tau) = \pi$. Because $\A$ respects $\P$, we have $\A(\tau') = \pi'$ and $\InTree(\tau') = \InTree(\tau)'$. We prove in the following that $\tau' \in \S$, $\S(\theta') = \tau'$ and $\InTree(\theta') = \InTree(\theta)'$ (see Claims $1$, $4$ and $3$ below). 

\smallskip

Because $\tau \in \S$, we may consider the canonical tree $T= \canT_{\tau}$ associated with $\tau$. Consider also its relabeling by $\lambda_\pi$, denoted $T'$. Of course, because $\post(T) = \tau$, we have $\post(T') = \tau'$. This will prove the first condition of $\A \circ \S$ respecting $\P$. 

\emph{Claim $1$:} $T'$ is decreasing. \\
It then follows from $\post(T') = \tau'$ and Corollary~\ref{COR:post_in_S} that $\tau' \in \S$. 

\emph{Proof of Claim $1$:} Consider any edge from a parent $b$ to a child $a$ in $T$, and the corresponding edge from $b'$ to $a'$ in $T'$. Because $T$ is decreasing, $a<b$ and all the elements that occur between $a$ and $b$ in $\tau = \post(T)$ are less than $b$. By Observation \ref{OBS:SR} all the elements between $a$ and $b$ in $\pi = \A(\tau)$ are less than $b$. Hence by Observation~\ref{OBS:P}, $b' > a'$, proving that $T'$ is decreasing.  

\emph{Claim $2$:} $T'$ is canonical. \\
It then follows from $\post(T') = \tau'$ and Proposition~\ref{PROP:unique_canonical_tree} that $T'=\canT_{\tau'}$ (the unique canonical tree associated with $\tau'$). 

\emph{Proof of Claim $2$:} Let $x$ be the left child of some vertex $z$ in $T$ and let $y$ be the leftmost element of the right subtree of $z$. Then $x > y$ since $T$ is canonical. Let $x'$ and $y'$ occupy the corresponding positions in $T'$. Since $x$ occurs immediately before $y$ in $\tau = \post(T)$,  $x' > y'$ (by Observations \ref{OBS:SR} and \ref{OBS:P}, as in the proof of Claim $1$).

\emph{Claim $3$:} $\InTree(\theta') = \InTree(\theta)'$.

\emph{Proof of Claim $3$:} 
Recall that $\S(\theta) = \tau$. By Proposition~\ref{PROP:all_preimages_from_canonical_tree}, $\InTree(\theta)$ is obtained from $T$ by applying some sequence of operations of the form:
\[ (\star) \left\{ 
\text{\begin{minipage}{0.95\textwidth}
Take a vertex $z$ with no left child, and one of its descendants $y$ on the leftmost branch of its right subtree. Remove the subtree rooted at $y$ and make it the left subtree of $z$.
\end{minipage}}
\right.
\]
Applying the same sequence of operations to $T'$ creates a tree with the same underlying structure as $\InTree(\theta)$, but with the labels arising from $T'$: this is the tree $\InTree(\theta)'$ and its in-order reading is $\inreading(\InTree(\theta)') = \theta'$. Because $T'$ is decreasing and since the operations $(\star)$ cannot create an increasing pair, $\InTree(\theta)'$ is a decreasing tree. Observation~\ref{OBS:Tin} then ensures that $\InTree(\theta') = \InTree(\theta)'$.

\emph{Claim $4$:} $\S(\theta') = \tau'$.

\emph{Proof of Claim $4$:} 
From Claim $3$, we know that $\InTree(\theta')$ is obtained from $T'$ by a sequence of operations $(\star)$. Moreover, from Claim $2$, $T'$ is the canonical tree of $\tau'$. Therefore, Proposition~\ref{PROP:all_preimages_from_canonical_tree} ensures that $\post(\InTree(\theta'))=\tau'$. Hence with Observation~\ref{OBS:stack_sorting_on_trees}, we deduce that $\S(\theta')=\post(\InTree(\theta'))=\tau'$.

\smallskip

To conclude the proof that $\A \circ \S$ respects $\P$, it remains to show that $\Phi_{\A \circ \S}: \theta \mapsto \theta'$ is a bijection between $(\A \circ \S)^{-1}(\pi)$ and $(\A \circ \S)^{-1}(\P(\pi)) = (\A \circ \S)^{-1}(\pi')$.

First, we claim that, for every $\tau \in \A^{-1}(\pi)$, the correspondence $\theta \mapsto \theta'$ is a bijective map between $\S^{-1}(\tau)$ and $\S^{-1}(\tau')$. 
This follows from Proposition~\ref{PROP:all_preimages_from_canonical_tree} (together with Observations~\ref{OBS:Tin} and~\ref{OBS:stack_sorting_on_trees}), as in the proofs of Claims $3$ and $4$ above. Details may also be found in the proof of Proposition~2.7 of~\cite{Bousquet:Sorted}.

Second, the set $(\A \circ \S)^{-1}(\pi)$ (resp. $(\A \circ \S)^{-1}(\pi')$) may be partitioned into the disjoint union of the sets $\S^{-1}(\tau)$ for $\tau \in \A^{-1}(\pi)$ (resp. $\S^{-1}(\tau')$ for $\tau' \in \A^{-1}(\pi')$). 
Because $A$ respects $\P$, the correspondence $\Phi_\A: \tau \mapsto \tau'$ is a bijection between $\A^{-1}(\pi)$ and $\A^{-1}(\pi')$. 

Hence the complete correspondence $\Phi_{\A \circ \S}: \theta \mapsto \theta'$ from $(\A \circ \S)^{-1}(\pi)$ to $(\A \circ \S)^{-1}(\pi')$ is a bijection, and $P$ respects $\A \circ \S$.
\end{proof}

Finally, let us observe that:
\begin{observation}
The identity operator respects $\P$.
\label{OBS:identity}
\end{observation}

\begin{proof}
When $\A$ is the identity, we have $\Phi_\A = \P$. So to show that the identity fulfills the definition of respecting $\P$, the only thing to prove it that $\InTree(\P(\pi)) = \lambda_\pi(\InTree(\pi))$ for any $\pi \in \Av(231)$. 
Because $\InTree(\pi)$ is decreasing, we deduce from Observation~\ref{OBS:P} that $\lambda_\pi(\InTree(\pi))$ is also decreasing. 
Moreover, the in-order reading of $\lambda_\pi(\InTree(\pi))$ is $\lambda_\pi \circ \pi = \P(\pi)$, since the one of $\InTree(\pi)$ is $\pi$. 
Observation~\ref{OBS:Tin} then gives $\lambda_\pi(\InTree(\pi)) = \InTree(\P(\pi))$. 
\end{proof}

Putting together Propositions~\ref{PROP:AoR_respects_P} and~\ref{PROP:AoS_respects_P} and Observation~\ref{OBS:identity}, we obtain our main theorem:

\begin{theorem}\label{THM:respects}
Every operator that is formed by composition from $\{\S, \R\}$ respects $\P$.
\end{theorem}

\begin{corollary}\label{COR:bijection}
For any composition $\A$ of operators from $\{\S,\R\}$, $\Phi_{\A}$ is a size-preserving bijection between the set of permutations sorted by $\S \circ \A$ and those sorted by $\S \circ \R \circ \A$.
\end{corollary}

\begin{proof}
A permutation $\theta$ is sorted by $\S \circ \A$ if and only if there exists $\pi \in \Av(231)$ such that $\theta \in \A^{-1}(\pi)$. From $\pi = \A(\theta)$, we have $\Phi_{\A}(\theta)= \lambda_{\pi} \circ \theta$. Because $\A$ respects $\P$, $\theta \in \A^{-1}(\pi)$ is equivalent to $\Phi_{\A}(\theta) \in \A^{-1}(\P(\pi))$. Finally, because $\P$ is a bijection between $\Av(231)$ and $\Av(132)$, the existence of $\pi \in \Av(231)$ such that $\Phi_{\A}(\theta) \in \A^{-1}(\P(\pi))$ is equivalent to the existence of $\tau \in \Av(132)$ such that $\Phi_{\A}(\theta) \in \A^{-1}(\tau)$, \emph{i.e.} to $\Phi_{\A}(\theta)$ being sorted to $\S \circ \R \circ \A$. 
\end{proof}

Corollary~\ref{COR:bijection} proves the first part of Conjecture~\ref{CONJ:BouvelGuibert}, namely that the number of permutations of each size sorted by $\S \circ \A$ and by $\S \circ \R \circ \A$ is the same.

We now study the properties of bijections $\Phi_{\A}$ in somewhat greater detail. This will prove the second part of Conjecture~\ref{CONJ:BouvelGuibert}, that deals with permutation statistics equidistributed over the set of permutations sorted by $\S \circ \A$ and the set of those sorted by $\S \circ \R \circ \A$.

\subsection{Statistics preserved by the bijections $\Phi_{\A}$}\label{SEC:stat}

As before, $\A$ denotes any composition of operators from $\{\S,\R\}$. 

\begin{theorem}
$\Phi_{\A}$ preserves the shape of the in-order tree.
\end{theorem}

\begin{proof}
From Theorem~\ref{THM:respects}, $\A$ respects $\P$. Hence for all permutations $\theta$ sorted by $\S \circ \A$, and denoting $\pi = \A(\theta)$, we have $\InTree(\Phi_{\A}(\theta)) = \lambda_{\pi}(\InTree(\theta))$, so that $\InTree(\Phi_{\A}(\theta))$ and $\InTree(\theta)$ have the same shape. 
\end{proof}

Because the shape of the in-order tree determines many permutation statistics (see Observation~\ref{OBS:stat_on_InTree} p.~\pageref{OBS:stat_on_InTree}), we have:

\begin{corollary}
$\Phi_{\A}$ preserves the following statistics: the number and positions of the right-to-left maxima, the number and positions of the left-to-right maxima and the up-down word (and hence also the many classical permutation statistics determined by the up-down word).
\end{corollary}

\begin{theorem}\label{THM:zeil}
If $\A = \A_0 \circ \S$ for some arbitrary composition $\A_0$ of operators from $\{\S,\R\}$, then $\Phi_{\A}$ preserves the Zeilberger statistic, defined as: $\zeil(\pi) = \max \{k \mid n (n-1) \ldots (n-k+1) \text{ is a subword of } \pi\}$.
In addition, if there is at least an operator $\S \circ \R$ in the composition that defines $\A_0$, then $\Phi_{\A}$ also preserves the reversal of the above statistics: $\Rzeil(\pi) = \max \{k \mid (n-k+1) \ldots (n-1) n \text{ is a subword of } \pi\}$.
\end{theorem}

\begin{proof}
Consider $\theta$ a permutation sorted by $\S \circ \A$, and set $\pi = \A(\theta)$. Notice that $\pi \in \Av(231)$. Writing $\P(\pi) = \lambda_{\pi} \circ \pi = \pi'$, and using the primed notation throughout as before, we have $\Phi_{\A}(\theta) =  \theta'$. Let $c \leq n$ be the smallest value such that for all $d \geq c$, $d' = d$. In what follows, we assume that $c \neq 1$, or the results follow trivially from $\Phi_{\A}(\theta) =  \theta$.

Set $k = \zeil(\theta)$. This means that the right branch from the root of $\InTree(\theta)$ is labeled by $n, n-1, \ldots, n-k+1$, and that the right child of the vertex labeled by $n-k+1$ (if it exists) is not labeled by $n-k$. Because $\InTree(\theta') = \InTree(\theta)'$, 
to show that $\zeil(\theta')=\zeil(\theta)$, it is enough to prove that the relabeling $'$ does not affect the elements larger than or equal to $n-k$, \emph{i.e.} that $c \leq n-k$.

Assume to the contrary $c>n-k$ holds. Then the post-order reading of $\InTree(\theta)$ gives $\S(\theta) = \sigma c (c+1) \ldots (n-1) n$, where $\sigma$ contains all values from $1$ to $(c-1)$. By assumption $\sigma$ contains at least two distinct elements $y=c-1$ and $x$ such that $x'=c-1$. We have $x<y$ but $y'<x'$. Since $x < y = c-1$, and all elements greater than or equal to $c$ occur as a suffix of $\S(\theta)$, Observation~\ref{OBS:SR} implies that there is no element larger than $c-1$ occurring between $x$ and $y$ in $\A_0 \circ \S(\theta) = \A(\theta) = \pi$. But then Observation~\ref{OBS:P} gives $x'<y'$, providing a contradiction and concluding the proof of the first statement of Theorem~\ref{THM:zeil}. 

Let us assume now that $\A_0$ contains at least one operator $\S \circ \R$, and let us write $\A = \B_0 \circ \S \circ \R \circ \S^k$, with $k \geq 1$. We also set $\S^k(\theta) = \tau$ and $\R(\tau) = \rho$. We have $\Phi_{\B_0 \circ \S}(\rho) = \rho'$.
The first statement of Theorem~\ref{THM:zeil} applied on $\B_0 \circ \S$ ensures that $\zeil(\rho) = \zeil(\rho')$. Most importantly, the proof of this statement also ensures that $\zeil(\rho) \leq n-c$. Hence, applying operator $\R$ gives $\Rzeil(\tau) \leq n-c$. It is simple to notice that for any permutation $\sigma$, we have $\Rzeil(\S(\sigma)) \geq \Rzeil(\sigma)$. In particular, we obtain $n-c \geq \Rzeil(\tau) = \Rzeil(\S^k(\theta)) \geq \Rzeil(\theta)$. Writing $k=\Rzeil(\theta)$, we then have $c \leq n-k$, so that no element of $\{n-k, n-k+1, \ldots, n\}$ is affected by the relabeling $'$. From this fact, we easily deduce that $\Rzeil(\theta') = \Rzeil(\theta)$.
\end{proof}

\subsection{Stating the main result}

Putting everything together, we have proved Conjecture~\ref{CONJ:BouvelGuibert}, namely:

\begin{theorem}
\label{THM:main_result}
For any operator $\A$ which is a composition of the operators $\S$ and $\R$, the number of permutations of each size sorted by $\S \circ \A$ and by $\S \circ \R \circ \A$ is the same. 

Moreover, the following permutation statistics are equidistributed across these two sets: number and positions of the right-to-left maxima, number and positions of the left-to-right maxima and up-down word (and hence also the many classical permutation statistics determined by the up-down word). To this list we may add the statistic $\zeil$ when $\A = \A_0 \circ \S$, and $\Rzeil$ when $\A = \B_0 \circ \S \circ \R \circ \S^k$ for some $k \geq 1$. 

More precisely, $\Phi_\A$ defines a size-preserving bijection between the set of permutations sorted by $\S \circ \A$ and the set of those sorted by $\S \circ \R \circ \A$ that preserves these statistics.
\end{theorem}

The statement of Theorem~\ref{THM:main_result} may in particular be considered for the operator $\A = \S$. This gives a size-preserving bijection $\Phi_\S$ between permutations sorted by $\S \circ \S$ and those sorted by $\S \circ \R \circ \S$. 
By means of generating trees,~\cite{Bouvel:Enumeration} implicitly defines another size-preserving bijection between these two sets. Both bijections preserves many permutation statistics, but we don't know whether they are actually two descriptions of the same bijection. 

\section{Wilf-equi\-va\-len\-ces derived from the bijection $\P$}\label{SEC:Wilf}

Recall a notation from the introduction: $\Av(B)$ is the set of all permutations that avoid simultaneously all the patterns in the set $B$. Such a set $\Av(B)$ is called a \emph{permutation class} (or \emph{class} for short), and $B$ is called its \emph{basis}. 

Two bases $B$ and $B'$ (or two classes $\Av(B)$ and $\Av(B')$) are said to be \emph{Wilf-equivalent} if $\Av(B)$ and $\Av(B')$ contain the same number of permutations of $[n]$ for every $n$. 
Coincidence of the enumeration sequences of two permutation classes (\emph{i.e.}, Wilf-equivalence) is a frequently observed phenomenon. The first (non trivial) example is that of $\Av(123)$ and $\Av(231)$, both enumerated by the Catalan numbers -- see~\cite{Simion:Restricted} for a bijective proof of this Wilf-equivalence. 
More examples of Wilf-equivalences may be found in~\cite{Kitaev} and references therein. 
One common form of Wilf-equivalence arises from symmetries of the avoidance relationship. For example, the reversal symmetry $\R$ provides a bijection between $\Av(231)$ and $\Av(132)$, proving that they are Wilf-equivalent. More generally, for any symmetry $\Z$ obtained composing reversal, complement and inverse, $\Av(\pi, \pi', \cdots, \pi'')$ and  $\Av(\Z(\pi), \Z(\pi'), \cdots, \Z(\pi''))$ are Wilf-equivalent, and we say that they are \emph{trivially} Wilf-equivalent. However, non trivial Wilf-equivalences are also somewhat common, and more interesting. 

In this section, we present some results showing how the bijection $\P$ from Section~\ref{SEC:P} furnishes a supply of Wilf-equivalences. 

\smallskip

Let us define two families of permutations $(\lambda_n)$ and $(\rho_n)$ recursively by $\lambda_1 = \rho_1 = 1$ and for all $n \geq 1$, $\lambda_{n+1} = 1 \ominus \rho_n$ and $\rho_{n+1} = \lambda_n \oplus 1$ -- see Figure~\ref{FIG:lambda_and_rho} for the diagrams of these permutations, and Section~\ref{SEC:P} for the definitions of $\oplus$ and $\ominus$. 
We also take the convention that $\lambda_0$ and $\rho_0$ denote the empty permutation $\varepsilon$. 
Of course, for every $n$, $\R(\lambda_n) = \rho_n$. 
Notice that for any $n$, $\lambda_n$ and $\rho_n$ are fixed by $\P$. 
Notice also that for any $n$, $\lambda_n$ is $\oplus$-indecomposable, \emph{i.e.} there are no non empty permutations $\alpha$ and $\beta$ such that $\lambda_n = \oplus[\alpha,\beta]$. 
Similarly, for any $n$, $\rho_n$ is $\ominus$-indecomposable.

\begin{figure}[h]
\begin{center}
$\lambda_n= \begin{array}{c}
\begin{tikzpicture}
\begin{scope}[scale=.35]
\draw[gray] (1,0) rectangle (4,3);
\draw (0,0) rectangle (4,4);
\draw (2.5,1.5) node {\small $\rho_{n-1}$};
\draw (0.5,3.5) [fill] circle (.25);
\end{scope} 
\end{tikzpicture}
\end{array}$ and $\rho_n= \begin{array}{c}
\begin{tikzpicture}
\begin{scope}[scale=.35]
\draw[gray] (0,0) rectangle (3,3);
\draw (0,0) rectangle (4,4);
\draw (1.5,1.5) node {\small $\lambda_{n-1}$};
\draw (3.5,3.5) [fill] circle (.25);
\end{scope} 
\end{tikzpicture}
\end{array}$; \qquad
$\lambda_6 = \begin{array}{c}
\begin{tikzpicture}
\begin{scope}[scale=.25]
\draw[gray!40] (3,0) rectangle (4,1);
\draw[gray!55] (2,0) rectangle (4,2);
\draw[gray!70] (2,0) rectangle (5,3);
\draw[gray!85] (1,0) rectangle (5,4);
\draw[gray] (1,0) rectangle (6,5);
\draw (0,0) rectangle (6,6);
\draw (0.5,5.5) [fill] circle (.25);
\draw (1.5,3.5) [fill] circle (.25);
\draw (2.5,1.5) [fill] circle (.25);
\draw (3.5,0.5) [fill] circle (.25);
\draw (4.5,2.5) [fill] circle (.25);
\draw (5.5,4.5) [fill] circle (.25);
\end{scope} 
\end{tikzpicture}
\end{array}$ and $\rho_6 = \begin{array}{c}
\begin{tikzpicture}
\begin{scope}[scale=.25]
\draw[gray!40] (2,0) rectangle (3,1);
\draw[gray!55] (2,0) rectangle (4,2);
\draw[gray!70] (1,0) rectangle (4,3);
\draw[gray!85] (1,0) rectangle (5,4);
\draw[gray] (0,0) rectangle (5,5);
\draw (0,0) rectangle (6,6);
\draw (0.5,4.5) [fill] circle (.25);
\draw (1.5,2.5) [fill] circle (.25);
\draw (2.5,0.5) [fill] circle (.25);
\draw (3.5,1.5) [fill] circle (.25);
\draw (4.5,3.5) [fill] circle (.25);
\draw (5.5,5.5) [fill] circle (.25);
\end{scope} 
\end{tikzpicture}
\end{array}$
\end{center}
\caption{Diagrams of $\lambda_n$ and $\rho_n$, for general $n$ and for $n=6$. \label{FIG:lambda_and_rho}}
\end{figure}
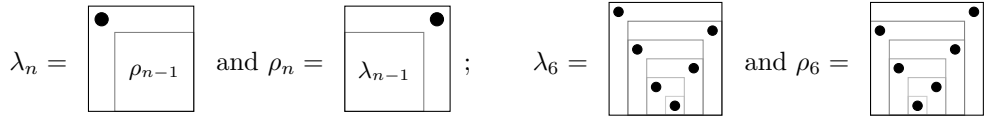

\begin{proposition}
\label{PROP:Wilf}
For every $n \geq 0$, and every $0\leq k \leq n-1$, letting $\pi=\lambda_k \oplus (1 \ominus \rho_{n-k-1})$, 
$\P$ is a size-preserving bijection between $\Av(231,\pi)$ and $\Av(132, \P(\pi))$. 
\end{proposition}

In the above statement and in what follows, our convention is that for $n=0$ and any $k$, $\lambda_k \oplus (1 \ominus \rho_{n-k-1})$ denotes the empty permutation $\varepsilon$. 

\begin{proof}
For any $n\geq 0$, let us denote by $\mathcal{P}(n)$ the following property: 
for every pattern $\pi$ of the form $\lambda_k \oplus (1 \ominus \rho_{n-k-1})$ with $0\leq k \leq n-1$, for every $\sigma \in \Av(231)$, if $\P(\sigma)$ contains $\P(\pi)$ then $\sigma$ contains $\pi$. 
We prove by induction that $\mathcal{P}(n)$ holds for all $n$. 

For $n=0$, \emph{i.e.} $\pi = \varepsilon$, the statement $\mathcal{P}(0)$ is clear. 
So assume that $n\geq 1$ and that $\mathcal{P}(\ell)$ holds for every $\ell \leq n-1$. 
To prove $\mathcal{P}(n)$, let us fix some pattern $\pi = \lambda_k \oplus (1 \ominus \rho_{n-k-1})$ with $0\leq k \leq n-1$. Notice that $\P(\pi) = (\lambda_k \oplus 1) \ominus \rho_{n-k-1}$.

We now prove by induction on $|\sigma|$ that for every $\sigma \in \Av(231)$, $\P(\sigma)$ contains $\P(\pi)$ implies that $\sigma$ contains $\pi$.
If $|\sigma|=1$, the above holds immediately. 
So consider $\sigma \in \Av(231)$ with $|\sigma| \geq 2$, assume that $\P(\sigma)$ contains $\P(\pi)$, and fix an occurrence of $\P(\pi)$ in $\P(\sigma)$.
Recall that we can write $\sigma =\alpha \oplus (1 \ominus \beta)$, which yields $\P(\sigma) = (\P(\alpha) \oplus 1) \ominus \P(\beta)$. 
We distinguish several cases according to how the occurrence of $\P(\pi)$ in $\P(\sigma)$ spreads over $(\P(\alpha) \oplus 1) \ominus \P(\beta)$. 
\begin{itemize}
\item
If $\P(\pi)$ occurs in $\P(\alpha) \oplus 1$ with $k\neq n-1$, then $\P(\pi)$ occurs in $\P(\alpha)$. 
By the induction hypothesis, $\alpha$ contains $\pi$, and so does $\sigma$.
\item
If $\P(\pi)$ occurs in $\P(\alpha) \oplus 1$ with $k= n-1$, then $\lambda_{n-1}$ occurs in $\P(\alpha)$.
By $\mathcal{P}(n-1)$, we obtain that $\alpha$ contains $\lambda_{n-1}$ so that $\sigma$ contains $\pi$.
\item
If $\P(\pi)$ occurs in $\P(\beta)$, then by the induction hypothesis $\pi$ occurs in $\beta$, hence in $\sigma$. 
\item
Otherwise, we can decompose $\P(\pi)$ as $\pi_1 \ominus \pi_2$ with both $\pi_1$ and $\pi_2$ not empty, $\pi_1$ occuring in $\P(\alpha) \oplus 1$ and $\pi_2$ occuring in $\P(\beta)$. 
But $\P(\pi) = (\lambda_k \oplus 1) \ominus \rho_{n-k-1}$, and because $\lambda_k \oplus 1$ and $\rho_{n-k-1}$ are $\ominus$-indecomposable, we necessarily have $\pi_1= \lambda_k \oplus 1$ and $\pi_2 =\rho_{n-k-1}$. 
Therefore, we deduce that $\P(\alpha)$ contains $\lambda_k$ and that $\P(\beta)$ contains $\rho_{n-k-1}$. From $\mathcal{P}(k)$ and $\mathcal{P}(n-k-1)$, 
we obtain that $\alpha$ contains $\lambda_k$ and that $\beta$ contains $\rho_{n-k-1}$, implying that $\sigma$ contains $\pi$.
\end{itemize}

This concludes the proof that $\mathcal{P}(n)$ holds for every $n\geq 0$. 
Following the same steps, it may be proved that: 
for every pattern $\pi$ of the form $\lambda_k \oplus (1 \ominus \rho_{n-k-1})$ with $0\leq k \leq n-1$, for every $\sigma \in \Av(231)$, if $\sigma$ contains $\pi$ then $\P(\sigma)$ contains $\P(\pi)$.
Details are left to the reader. 
We conclude that 
for every pattern $\pi$ of the form $\lambda_k \oplus (1 \ominus \rho_{n-k-1})$ with $0\leq k \leq n-1$, 
for every $\sigma \in \Av(231)$, $\sigma$ contains $\pi$ if and only if $\P(\sigma)$ contains $\P(\pi)$. In other words $\sigma \in \Av(231, \pi)$ if and only if $\P(\sigma) \in \Av(132,\P(\pi))$, proving the announced statement. 
\end{proof}

Although this is not our main point here, it is worth noticing that there is a converse statement to Proposition~\ref{PROP:Wilf}. Namely: 

\begin{proposition}
Let $\pi$ be a $231$-avoiding permutation of size $n$. 
$\P$ is a size-preserving bijection between $\Av(231,\pi)$ and $\Av(132, \P(\pi))$ 
if and only if $\pi=\lambda_k \oplus (1 \ominus \rho_{n-k-1})$, for some $0\leq k \leq n-1$. 
\label{PROP:Wilf_converse}
\end{proposition}

The proof of Proposition~\ref{PROP:Wilf_converse} makes use of the three observations that follow. 

\begin{observation}\label{OBS:stable}
Let $\pi = \alpha \oplus ( 1 \ominus \beta) \in \Av(231)$ be a permutation such that $\P$ is a bijection between $\Av(231,\pi)$ and $\Av(132, \P(\pi))$. Then the same holds for $\alpha$ and $\beta$ instead of $\pi$.
\end{observation}

\begin{proof}
Let us assume that $\P$ is not a bijection between $\Av(231,\alpha)$ and $\Av(132, \P(\alpha))$ (resp. $\Av(231,\beta)$ and $\Av(132, \P(\beta))$). Then one of the following holds: either there exists $\sigma \in \Av(231,\alpha)$ (resp.~$\sigma \in \Av(231,\beta)$) such that $\P(\sigma)$ contains $\P(\alpha)$ (resp.~$\P(\beta)$), or there exists $\sigma \in \Av(231)$ such that $\sigma$ contains $\alpha$ (resp.~$\beta$) but $\P(\sigma)$ avoids $\P(\alpha)$ (resp.~$\P(\beta)$). Consider the permutation $\tau = \sigma \oplus (1 \ominus \beta)$ (resp.~$\tau = \alpha \oplus (1 \ominus \sigma)$) and its image by $\P$: $\P(\tau)=(\P(\sigma) \oplus 1) \ominus \P(\beta)$ (resp.~$\P(\tau)=(\P(\alpha) \oplus 1) \ominus \P(\sigma)$).
In the first case, $\tau$ avoids $\pi$ but $\P(\tau)$ contains $\P(\pi)$, 
and in the second case, $\tau$ contains $\pi$ but $\P(\tau)$ avoids $\P(\pi)$. 
This is a contradiction to $\P$ being a bijection between $\Av(231,\pi)$ and $\Av(132, \P(\pi))$, and concludes the proof. 
\end{proof}

\begin{observation}\label{OBS:alpha}
Let $\pi = \alpha \oplus ( 1 \ominus \beta) \in \Av(231)$ be a permutation such that $\P$ is a bijection between $\Av(231,\pi)$ and $\Av(132, \P(\pi))$. 
Then $\alpha$ begins with its maximum.
\end{observation}

\begin{proof}
Let us assume that $\alpha$ does not begin with its maximum. As $\alpha$ avoids $231$ we can write $\alpha=\gamma \oplus (1 \ominus \delta)$ with $\gamma$ non empty. From $\pi = \gamma \oplus (1 \ominus \delta) \oplus ( 1 \ominus \beta)$ we deduce that 
$\P(\pi) = \left( \left( \left( \P(\gamma) \oplus 1 \right) \ominus \P(\delta) \right) \oplus 1\right) \ominus \P(\beta)$.
Consider now the permutation 
$ \sigma = \gamma \oplus \left( 1 \ominus \left( \left(1 \ominus \delta \right) \oplus (1 \ominus \beta) \right) \right)$ 
and its image under $\P$: 
$ \P(\sigma) = (\P(\gamma) \oplus 1) \ominus \Big( \big( (1 \ominus \P(\delta) ) \oplus 1\big) \ominus \P(\beta)\Big)$. 
Then $\sigma$ contains $\pi$ but $\P(\sigma)$ avoids $\P(\pi)$ (see Figure~\ref{FIG:OBS_alpha}). This contradicts that $\P$ is a bijection between $\Av(231,\pi)$ and $\Av(132, \P(\pi))$, and ensures that $\alpha$ must begin with its maximum. 
\end{proof}

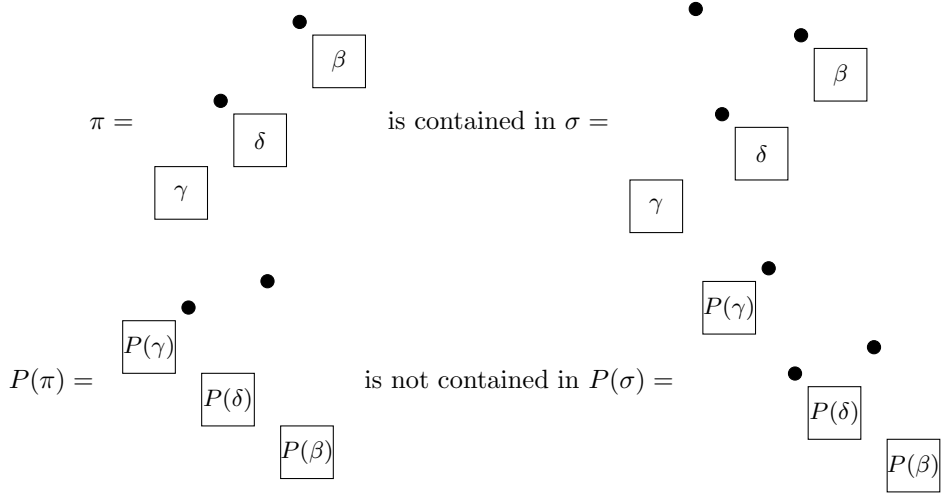
\begin{figure}[ht]
\begin{center}
$\pi= \begin{array}{c}
\begin{tikzpicture}
\begin{scope}[scale=.35]
\draw (0,0) rectangle (2,2);
\draw (1,1) node {\small $\gamma$};
\draw (2.5,4.5) [fill] circle (.25);
\draw (3,2) rectangle (5,4);
\draw (4,3) node {\small $\delta$};
\draw (5.5,7.5) [fill] circle (.25);
\draw (6,5) rectangle (8,7);
\draw (7,6) node {\small $\beta$};
\end{scope} 
\end{tikzpicture}
\end{array}$ is contained in 
$\sigma= \begin{array}{c}
\begin{tikzpicture}
\begin{scope}[scale=.35]
\draw (0,0) rectangle (2,2);
\draw (1,1) node {\small $\gamma$};
\draw (2.5,8.5) [fill] circle (.25);
\draw (3.5,4.5) [fill] circle (.25);
\draw (4,2) rectangle (6,4);
\draw (5,3) node {\small $\delta$};
\draw (6.5,7.5) [fill] circle (.25);
\draw (7,5) rectangle (9,7);
\draw (8,6) node {\small $\beta$};
\end{scope}
\end{tikzpicture}
\end{array}$

\medskip

$\P(\pi)= \begin{array}{c}
\begin{tikzpicture}
\begin{scope}[scale=.35]
\draw (0,4) rectangle (2,6);
\draw (1,5) node {\small $\P(\gamma)$};
\draw (2.5,6.5) [fill] circle (.25);
\draw (3,2) rectangle (5,4);
\draw (4,3) node {\small $\P(\delta)$};
\draw (5.5,7.5) [fill] circle (.25);
\draw (6,0) rectangle (8,2);
\draw (7,1) node {\small $\P(\beta)$};
\end{scope}
\end{tikzpicture}
\end{array}$ is not contained in 
$\P(\sigma)= \begin{array}{c}
\begin{tikzpicture}
\begin{scope}[scale=.35]
\draw (0,6) rectangle (2,8);
\draw (1,7) node {\small $\P(\gamma)$};
\draw (2.5,8.5) [fill] circle (.25);
\draw (3.5,4.5) [fill] circle (.25);
\draw (4,2) rectangle (6,4);
\draw (5,3) node {\small $\P(\delta)$};
\draw (6.5,5.5) [fill] circle (.25);
\draw (7,0) rectangle (9,2);
\draw (8,1) node {\small $\P(\beta)$};
\end{scope}
\end{tikzpicture}
\end{array}$
\end{center}
\caption{Permutations $\pi$, $\P(\pi)$, $\sigma$ and $\P(\sigma)$ in the proof of Observation~\ref{OBS:alpha}. \label{FIG:OBS_alpha}}
\end{figure}

\begin{observation}\label{OBS:beta}
Let $\pi = \alpha \oplus ( 1 \ominus \beta) \in \Av(231)$ be a permutation such that $\P$ is a size-preserving bijection between $\Av(231,\pi)$ and $\Av(132, \P(\pi))$. 
Then $\beta$ ends with its maximum.
\end{observation}

\begin{proof}
The proof is similar of that of Observation~\ref{OBS:alpha}, assuming that $\beta=\gamma \oplus (1 \ominus \delta)$ does not end with its maximum, and considering the permutation $\sigma$ shown in Figure~\ref{FIG:OBS_beta}. 
\end{proof}

\begin{figure}[ht]
\begin{center}$\pi= \begin{array}{c}
\begin{tikzpicture}
\begin{scope}[scale=.35]
\draw (0,0) rectangle (2,2);
\draw (1,1) node {\small $\alpha$};
\draw (2.5,7.5) [fill] circle (.25);
\draw (3,2) rectangle (5,4);
\draw (4,3) node {\small $\gamma$};
\draw (5.5,6.5) [fill] circle (.25);
\draw (6,4) rectangle (8,6);
\draw (7,5) node {\small $\delta$};
\end{scope} 
\end{tikzpicture}
\end{array}$ is not contained in 
$\sigma= \begin{array}{c}
\begin{tikzpicture}
\begin{scope}[scale=.35]
\draw (0,0) rectangle (2,2);
\draw (1,1) node {\small $\alpha$};
\draw (2.5,5.5) [fill] circle (.25);
\draw (3,2) rectangle (5,4);
\draw (4,3) node {\small $\gamma$};
\draw (5.5,4.5) [fill] circle (.25);
\draw (6.5,8.5) [fill] circle (.25);
\draw (7,6) rectangle (9,8);
\draw (8,7) node {\small $\delta$};
\end{scope}
\end{tikzpicture}
\end{array}$

\medskip

$\P(\pi)= \begin{array}{c}
\begin{tikzpicture}
\begin{scope}[scale=.35]
\draw (0,5) rectangle (2,7);
\draw (1,6) node {\small $\P(\alpha)$};
\draw (2.5,7.5) [fill] circle (.25);
\draw (3,2) rectangle (5,4);
\draw (4,3) node {\small $\P(\gamma)$};
\draw (5.5,4.5) [fill] circle (.25);
\draw (6,0) rectangle (8,2);
\draw (7,1) node {\small $\P(\delta)$};
\end{scope}
\end{tikzpicture}
\end{array}$ is contained in 
$\P(\sigma)= \begin{array}{c}
\begin{tikzpicture}
\begin{scope}[scale=.35]
\draw (0,5) rectangle (2,7);
\draw (1,6) node {\small $\P(\alpha)$};
\draw (2.5,7.5) [fill] circle (.25);
\draw (3,2) rectangle (5,4);
\draw (4,3) node {\small $\P(\gamma)$};
\draw (5.5,4.5) [fill] circle (.25);
\draw (6.5,8.5) [fill] circle (.25);
\draw (7,0) rectangle (9,2);
\draw (8,1) node {\small $\P(\delta)$};
\end{scope}
\end{tikzpicture}
\end{array}$
\end{center}
\caption{Permutations $\pi$, $\P(\pi)$, $\sigma$ and $\P(\sigma)$ in the proof of Observation~\ref{OBS:beta}. \label{FIG:OBS_beta}}
\end{figure}
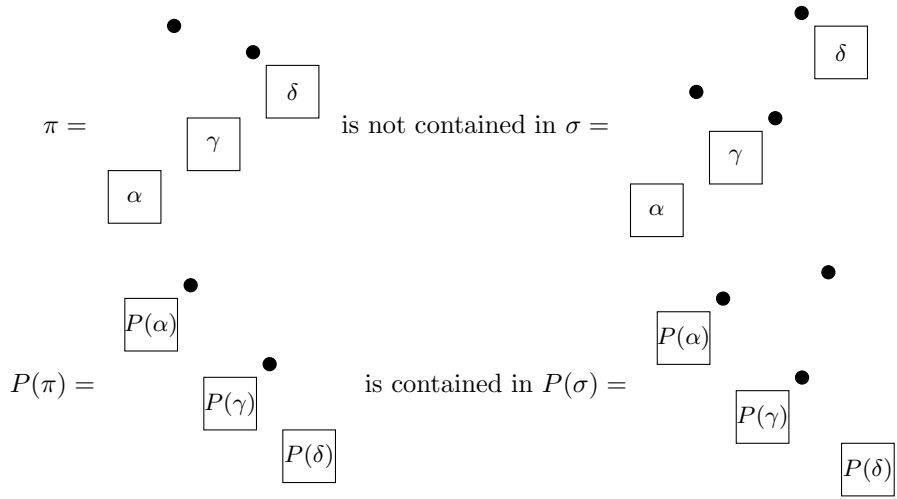

\begin{proof}[Proof of Proposition~\ref{PROP:Wilf_converse}] 
With Proposition~\ref{PROP:Wilf}, we are left to prove that 
for any $\pi \in \Av(231)$, 
if $\P$ is a size-preserving bijection between $\Av(231,\pi)$ and $\Av(132, \P(\pi))$,
then $\pi=\lambda_k \oplus (1 \ominus \rho_{n-k-1})$, for some $0\leq k \leq n-1$. 

Consider such a permutation $\pi$, write $\pi = \alpha \oplus (1 \ominus \beta)$, and set $k=|\alpha|$. 
Observations~\ref{OBS:alpha} and~\ref{OBS:beta} ensure that $\alpha$ begins with its maximum and that $\beta$ ends with its maximum. 
Observation~\ref{OBS:stable} also ensures that $\P$ is a bijection between $\Av(231,\alpha)$ and $\Av(132, \P(\alpha))$ (resp. $\Av(231,\beta)$ and $\Av(132, \P(\beta))$). 
Hence, to conclude the proof of Proposition~\ref{PROP:Wilf_converse}, it is enough to prove that for every permutation $\sigma$ such that $\P$ is a bijection between $\Av(231,\sigma)$ and $\Av(132, \P(\sigma))$, if $\sigma$ starts (resp.~ends) with its maximum, then $\sigma = \lambda_{\ell}$ (resp.~$\rho_{\ell}$) for some $\ell$. 
This is obtained proving by induction the following statement: 
for every $\ell \geq 1$, for every $\sigma \in \Av(231)$ of size $\ell$, if $\P$ is a bijection between $\Av(231,\sigma)$ and $\Av(132, \P(\sigma))$ then $\sigma(\ell)=\ell $ implies $ \sigma = \rho_{\ell}$ and  $\sigma(1)=\ell $ implies $ \sigma = \lambda_{\ell}$. 
This is clear for $\ell =1$, so take $\ell \geq 2$ and assume the above statement holds for $\ell-1$. 
Consider $\sigma \in \Av(231)$ of size $\ell$ such that $\P$ is a bijection between $\Av(231,\sigma)$ and $\Av(132, \P(\sigma))$. 
If $\sigma(\ell)=\ell$ (resp.~$\sigma(1)=\ell$), then $\sigma = \tau \oplus 1$ (resp.~$\sigma = 1 \ominus \tau$). 
By Observations~\ref{OBS:stable} and~\ref{OBS:alpha} (resp.~\ref{OBS:beta}), $\P$ is a bijection between $\Av(231,\tau)$ and $\Av(132, \P(\tau))$ and $\tau$ starts (resp.~ends) with its maximum. 
By the inductive hypothesis, $\tau = \lambda_{\ell-1}$ (resp.~$\tau = \rho_{\ell-1}$), so we deduce that $\sigma = \rho_{\ell}$ (resp.~$\sigma = \lambda_{\ell}$). 
\end{proof}

A consequence of Proposition~\ref{PROP:Wilf} is that 
for all $n$ and $0\leq k \leq n-1$, $\Av(231,\lambda_k \oplus (1 \ominus \rho_{n-k-1}))$ and $\Av(231, \lambda_{n-k-1} \oplus (1 \ominus \rho_{k}))$ are Wilf-equivalent. 
Indeed, letting $\pi=\lambda_k \oplus (1 \ominus \rho_{n-k-1})$, $\R \circ \P$ provides a size-preserving bijection from $\Av(231,\pi)$ to $\Av(231, \R(\P(\pi)))$, and $\R(\P(\pi))= \lambda_{n-k-1} \oplus (1 \ominus \rho_k)$. 
For every $n\geq 1$, Proposition~\ref{PROP:Wilf} therefore produces $n$ Wilf-equivalences (although, with some redundancies) between pairs of permutation classes, both of the form $\Av(231,\tau)$ with $\tau$ avoiding $231$. 

Moreover, as we explain in~\cite{FPSAC2013}, it is possible to compute the generating function of $\Av(231,\pi)$, for any $\pi=\lambda_k \oplus (1 \ominus \rho_{n-k-1})$. 
Regardless of $k$, this generating function is $F_n$ defined recursively by $F_1(t) = 1$, and for $n \geq 1, F_{n+1}(t) = \frac{1}{1-t F_n(t)}$ for $n \geq 1$. 
Therefore, for any fixed $n$, all classes $\Av(231,\lambda_k \oplus (1 \ominus \rho_{n-k-1}))$ are Wilf-equivalent. 
The proof of this result is analytic, and does not allow to find a bijection between any two such classes. 
In~\cite{Mansour:Restricted,Mansour:Chebyshev}, the authors are also interested in the enumeration of permutation classes of the form $\Av(231,\tau)$ -- or rather their reversal $\Av(132,\R(\tau))$. They show in particular that the generating function of a class $\Av(231,\tau)$ is also $F_n$, when $\tau$ is (the reversal of) a \emph{layered} permutation of size $n$ with two layers, or a \emph{wedge} permutation of size $n$. 
The proof is also analytic, and the authors indicate that it would be very interesting to find a bijective proof of these results. 

In the forthcoming paper~\cite{futurework}, we provide a unified proof and a generalization of both these results. 
More precisely, we describe a family $\bigclass$ of permutations such that for any $\tau$ and $\tau'$ of the same size $n$ in $\bigclass$, $\Av(231,\tau)$ and $\Av(231,\tau')$ are Wilf-equivalent, and their generating function is $F_n$. 
The family $\bigclass$ contains (reversals of) layered permutations with two layers, (reversals of) wedge permutations, and permutations of the form $\lambda_k \oplus (1 \ominus \rho_{n-k-1})$, but also much more: it actually contains $M_n$ permutations of any size $n$, where $M_n$ is the $n$-th Motzkin number. 
Unlike previous analytical proofs, our proof has a very combinatorial flavor, 
and allows the derivation of bijections between $\Av(231, \tau)$ and $\Av(231, \tau')$ for any $\tau$ and $\tau'$ of the same size in~$\bigclass$. 

\bigskip

\paragraph*{Acknowledgements} 

Many thanks to Olivier Guibert whose experimental investigations were critical in framing the questions addressed in this paper. Much of the work in this paper was supported by the software suite \emph{PermLab} \cite{PermLab1.0}, and extensions of it. Michael Albert would like to thank LaBRI for their hospitality and support in August and September of 2012 when most of this work was carried out.

\end{document}